\documentclass[a4paper,12pt,reqno]{amsart}
\usepackage[english]{babel}
\usepackage{amsmath}
\usepackage{amscd}
\usepackage{amsgen}
\usepackage[final]{epsfig}
\usepackage{latexsym}
\usepackage{amsfonts}
\usepackage{amssymb}
\usepackage{amsthm}
\usepackage{url}
\usepackage{slashed}

\allowdisplaybreaks

\catcode`\@=11
\@namedef{subjclassname@2010}{%
\textup{2010} Mathematics Subject Classification} \catcode`\@=12

\newtheorem{thm}{Theorem}[section]
\newtheorem{lemm}{Lemma}[section]

\newtheorem{exs}{Examples}[section]

\newtheorem{rem}{Remark}[section]

\newcommand{\Ric}{\operatorname{Ric}}
\newcommand{\scal}{\operatorname{scal}}
\newcommand{\tr}{\operatorname{tr}}

\def\J{{\sf J}}
\def\S{{\mathbb S}}        
\def\B{{\mathcal B}}       
\def\CC{{\bf C}}           
\def\G{{\mathcal G}}       
\def\M{{\mathcal M}}
\def\Q{{\mathcal Q}}       
\def\Rho{{\sf P}}          
\def\T{{\mathcal T}}       

\def\f{{\frac{n}{2}}}
\def\st{\stackrel{\text{def}}{=}}

\numberwithin{equation}{section}

\title{On the recursive structure of Branson's $Q$-curvature}

\author{Andreas Juhl}

\address{Humboldt-Universit\"at, Institut f\"ur Mathematik,
Unter den Linden, D-10099 Berlin, Germany}

\email{ajuhl@math.hu-berlin.de}

\address{University Uppsala, Department of Mathematics, P.O. Box 480,
S-75106 Uppsala, Sweden}

\email{andreasj@math.uu.se}

\begin{document}

\begin{abstract} We prove universal recursive formulas for Branson's
$Q$-curvatures in terms of respective lower-order $Q$-curvatures,
lower-order GJMS-operators and holographic coefficients.
\end{abstract}


\maketitle

\centerline \today

\footnotetext{The work was supported by SFB 647 "Raum-Zeit-Materie"
of DFG.}

\tableofcontents


\section{Introduction and formulation of the main result}

On any Riemannian manifold $(M,g)$ of {\em even} dimension $n$,
there is a finite sequence $P_2(g), P_4(g), \dots, P_n(g)$ of
geometric differential operators of the form\footnote{We use the
convention that $-\Delta \ge 0$}
$$
\Delta_g^N + \mbox{lower order terms}
$$
which are conformally covariant in the sense that
\begin{equation}\label{covar}
e^{(\f+N)\varphi} P_{2N}(e^{2\varphi}g)(u) =
P_{2N}(g)(e^{(\f-N)\varphi}u)
\end{equation}
for all $\varphi \in C^\infty(M)$. Similarly, on manifolds of {\em
odd} dimension there is an infinite sequence $P_2(g), P_4(g), \dots$
of geometric operators satisfying \eqref{covar}. The operators
$P_{2N}(g)$ are {\em geometric} in the sense that the lower order
terms are determined by the metric and its curvature. These
operators were constructed in the seminal work \cite{GJMS}. They
will be referred to as the GJMS-operators.

The constant terms of the GJMS-operators lead to the notion of
Branson's $Q$-curvatures (see \cite{bran-2}). In fact, for $2N < n$
it is natural to write the constant term of $P_{2N}$ in the form
\begin{equation}\label{Q-curv}
P_{2N}(g)(1) = (-1)^N \left(\f-N\right) Q_{2N}(g)
\end{equation}
with a scalar Riemannian curvature invariant $Q_{2N}(g) \in
C^\infty(M)$ of order $2N$. For even $n$, the {\em critical}
GJMS-operator $P_n$ has vanishing constant term and \eqref{Q-curv}
cannot be used to define an analogous quantity $Q_n$. However, $Q_n$
can be defined through $Q_{2N}$ for $2N<n$ by a continuation in the
dimension. The quantities $Q_{2N}$ will be called Branson's
$Q$-curvatures. For even $n$, $Q_n$ will be called the {\em
critical} $Q$-curvature.

The following two special cases are well-known. We have
$$
Q_2 = \frac{\scal}{2(n-1)}
$$
and
\begin{equation}\label{paneitz}
Q_4 = \f \J^2 - 2 |\Rho|^2 - \Delta (\J), \; n \ge 3,
\end{equation}
where we used the abbreviations
$$
\J = \frac{\scal}{2(n-1)} \quad \mbox{and} \quad \Rho =
\frac{1}{n-2}(\Ric-\J g).
$$
$\Rho$ is the Schouten tensor. The quantities $Q_2$ and $Q_4$ appear
in the corresponding Yamabe and Paneitz operators
$$
P_2 = \Delta - \left(\f\!-\!1\right) Q_2
$$
and
$$
P_4 = \Delta^2 + \delta ((n\!-\!2)\J - 4 \Rho) d +
\left(\f\!-\!2\right) Q_4.
$$

The main purpose of this paper is to establish formulas for all
higher order $Q$-curvatures.

In order to formulate the main result, we need some more notation.

First, a sequence $I=(I_1,\dots,I_r)$ of integers $I_j \ge 1$ will
be regarded as a composition of the sum $|I| = I_1 + I_2 + \cdots +
I_r$, where two representations which contain the same summands but
differ in the order of the summands are regarded as different. $|I|$
will be called the size of $I$. For $I=(I_1,\dots,I_r)$, we set
$$
P_{2I} = P_{2I_1} \circ \cdots \circ P_{2I_r}.
$$
We define the multiplicity $m_I$ of the composition $I$ by
\begin{equation}\label{m-form}
m_I = -(-1)^r |I|! \, (|I|\!-\!1)! \prod_{j=1}^r \frac{1}{I_j! \,
(I_j\!-\!1)!} \prod_{j=1}^{r-1} \frac{1}{I_j \!+\! I_{j+1}}.
\end{equation}
Here, an empty product has to be interpreted as $1$. Note that
$m_{(N)} = 1$ for all $N \ge 1$ and
$$
\sum_{|I|=N} m_I = 0
$$
(see Lemma 2.1 in \cite{juhl-power}). In these terms, we introduce
the generating function
\begin{equation}\label{G}
\G(r) = 1 + \sum_{N \ge 1} \left( \sum_{a+|J|=N} m_{(J,a)} (-1)^a
P_{2J}(Q_{2a}) \right) \frac{r^N}{N!(N\!-\!1)!}.
\end{equation}
Here, for even $n$, the sum on the right-hand side is to be
understood as a finite sum over $1 \le N \le n$.

A second ingredient of our formula for $Q$-curvatures comes from
Poincar\'e-Einstein metrics. Let $n$ be even. For a given metric $g$
on the manifold $M$ of dimension $n$, let
\begin{equation}\label{PE1}
g_+ = r^{-2} (dr^2 + g_r)
\end{equation}
with
\begin{equation}\label{PE2}
g_r = g + r^2 g_{(2)} + \dots + r^{n-2} g_{(n-2)} + r^n (g_{(n)}+
\log r \bar{g}_{(n)}) + \cdots
\end{equation}
be a metric on $X = (0,\varepsilon) \times M$ so that the tensor
$\Ric(g_+) + n g_+$ satisfies the Einstein condition
\begin{equation}\label{einstein}
\Ric(g_+) + n g_+ = O(r^{n-2})
\end{equation}
together with a certain vanishing trace condition. These conditions
uniquely determine the coefficients $g_{(2)}, \dots, g_{(n-2)}$ and
the quantity $\tr_g(g_{(n)})$. They are given as polynomial formulas
in terms of $g$, its inverse, the curvature tensor of $g$, and its
covariant derivatives. A metric $g_+$ with these properties is
called a Poincar\'e-Einstein metric with conformal infinity $[g]$.
Similarly, for odd $n$, the Einstein condition determines all
coefficients in the formal power series
$$
g_r = g + r^2 g_{(2)} + r^4 g_{(4)} + \cdots
$$
with only even powers of $r$. For full details see \cite{FG-final}.
The volume form of $g_+$ can be written as
$$
vol(g_+) = r^{-n-1} v(r) dr vol(g),
$$
where
$$
v(r) = vol(g_r)/vol(g) \in C^\infty(M).
$$
The coefficients in the Taylor series
$$
v(r) = v_0 + v_2 r^2 + v_4 r^4 + \cdots
$$
are known as the {\em renormalized volume} coefficients
\cite{G-vol}, \cite{G-ext} or {\em holographic} coefficients
\cite{juhl-book}. The coefficient $v_{2j} \in C^\infty(M)$ is given
by a local formula which involves at most $2j$ derivatives of the
metric.

The following theorem is the main result of the present paper. It
settles Conjecture 9.2 in \cite{juhl-power}.

\begin{thm}\label{Q-GF} On any Riemannian manifold $M$ of dimension
$n \ge 3$,
\begin{equation}\label{Q-G}
\G \left( \frac{r^2}{4} \right) = \sqrt{v(r)}.
\end{equation}
\end{thm}

Some comments are in order. The relation \eqref{Q-G} is to be
understood as an identity of formal power series in $r$. Moreover,
for even $n$, it is to be understood as an identity of finite power
series terminating at $r^n$. The Taylor coefficients of
$$
w(r) = \sqrt{v(r)} = 1 + w_2 r^2 + w_4 r^4 + \cdots
$$
can be expressed in terms of the holographic coefficients $v_2, v_4,
\cdots$. In particular, we have
\begin{align*}
2 w_2 & = v_2, \\
8 w_4 & = 4 v_4 - v_2^2, \\
16 w_6 & = 8 v_6 - 4 v_4 v_2 + v_2^3, \\
128 w_8 & = 64 v_8 - 32 v_6 v_2 - 16 v_4^2 + 24 v_2^2 v_4 - 5 v_2^4.
\end{align*}
Note that
\begin{equation}\label{v2v4}
-2 v_2 = \J \quad \mbox{and} \quad 8 v_4 = \J^2 - |\Rho|^2.
\end{equation}
Graham \cite{G-ext} describes an algorithm to derive formulas for
holographic coefficients in terms of the metric, and displays
explicit formulas for $v_6$ and $v_8$. In particular,
$$
- 48 v_6 = 6 \tr (\wedge^3 \Rho) - 2 (\Omega^{(1)},\Rho),
$$
where $\Omega^{(1)}$ denote Graham's first extended obstruction
tensor.\footnote{In more familiar terms, $\Omega^{(1)}$ equals $
\B/4-n$, where $\B$ is a version of the Bach tensor in dimension
$n$.} We refer to \cite{juhl-book} for the details of such
calculations. Note also that for locally conformally flat metrics,
\begin{equation}\label{c-flat}
(-2)^N v_{2N} = \tr (\wedge^N \Rho).
\end{equation}

Theorem \ref{Q-GF} provides recursive formulas for $Q_{2N}$ in the
following way. For any $N \ge 1$ (so that $2N \le n$ if $n$ is
even), \eqref{Q-G} states that
\begin{equation}\label{equiv}
\sum_{a+|J|=N} m_{(J,a)} (-1)^a P_{2J}(Q_{2a}) = 2^{2N} N!
(N\!-\!1)! w_{2N}.
\end{equation}
One of the $2^{N-1}$ items in the sum on the left-hand side of
\eqref{equiv} is $(-1)^N Q_{2N}$. All other items are defined in
terms of lower-order GJMS-operators acting on lower-order
$Q$-curvatures.

The relation \eqref{equiv} can be regarded as a formula for the
difference
$$
Q_{2N} - (-1)^N 2^{2N-1} N! (N\!-\!1)! v_{2N}.
$$
An alternative formula for the same difference is given by the {\em
holographic} formula
$$
2N c_N Q_{2N} = \sum_{j=0}^{N-1} (N\!-\!j)
\T_{2j}^*\left(\f\!-\!N\right)(v_{2N-2j});
$$
for the notation see Section \ref{curv}. In the critical case
$2N=n$, the latter formula was proved in \cite{gj}. For the general
case we refer to \cite{master}.

The identities \eqref{Q-G} are valid in {\em all} dimensions. This
feature will be referred to as {\em universality}.

For the convenience of the reader, we display the explicit formulas
for the four lowest order $Q$-curvatures. We find
\begin{align*}
Q_2 & = - 4 w_2, \\
Q_4 & = - P_2(Q_2) + 2^4 2! w_4, \\
Q_6 & = -2P_2(Q_4) + 2 P_4(Q_2) - 3 P_2^2(Q_2) - 2^6 3! 2! w_6
\end{align*}
and
\begin{multline}\label{Q8}
Q_8 = -3P_2(Q_6) - 3 P_6(Q_2) + 9P_4(Q_4) \\ + 8 P_2P_4(Q_2) - 12
P_2^2(Q_4) + 12 P_4P_2(Q_2) - 18 P_2^3(Q_2) + 2^8 4! 3! w_8.
\end{multline}
By $P_2 = \Delta-(\f\!-\!1)\J$ and \eqref{v2v4}, the formula for
$Q_4$ is easily seen to be equivalent to \eqref{paneitz}. Similarly,
combining formula \eqref{Q8} with recursive formulas for $P_4$ and
$P_6$ yields more explicit presentations of $Q_8$. For the details
concerning such consequences we refer to \cite{Q8}.

For round spheres $\S^n$, Theorem \ref{Q-GF} was proved in
\cite{juhl-power} by direct summation of the left-hand side. The
expectation that Theorem \ref{Q-GF} holds true also for
pseudo-Riemannian metrics is supported by the summation formulas on
$\S^{p,q}$ proved in \cite{JK}.

Finally, we emphasize that the variable $r$ plays different roles on
both sides of \eqref{Q-G}. In fact, $r$ is used as a formal variable
of a generating function on the left-hand side and as a defining
function on the right-hand side.

The paper is organized as follows. In Section \ref{poly}, we
establish explicit formulas for a sequence of recursively defined
operator-valued polynomials $\pi_{2N}(\lambda)$. These are closely
related to the $Q$-curvature polynomials $Q_{2N}^{res}(\lambda)$
that were introduced in \cite{juhl-book}. In Section \ref{curv}, we
recall this concept and show that the relation implies Theorem
\ref{Q-GF}, when combined with a result of \cite{Q8}. In Section
\ref{sum}, we prove a summation formula for GJMS-operators on round
spheres which is parallel to a summation formula for GJMS-operators
proved in \cite{juhl-power}.

We are grateful to Carsten Falk for his help with computer
experiments in the early stages of the work. The results were
presented at the conference ``Geometric Scattering Theory and
Applications'' at the Banff International Research Station (March,
2010).

\section{The polynomials $\pi_{2N}(\lambda)$}\label{poly}

In the present section, we discuss a sequence of operator-valued
polynomials which are closely related to the $Q$-curvature
polynomials. That relation will be important in Section \ref{curv}.

We start by defining some higher analogs of the multiplicities
$m_I$. We set $m_I^{(1)} = m_I$, and define the rational numbers
$m_I^{(k)}$ for $k \ge 2$ by the formulas
\begin{equation}\label{m-a}
m_{(a,J)}^{(k)} = \frac{\sum_{j=0}^{k-1} s(N,N-j) |J|^{k-1-j}}{(N-1)
\cdots (N-k+1)} \, m_{(a,J)}^{(1)}
\end{equation}
if $a+|J|=N$ and $2 \le k \le N-1$, and
\begin {equation}\label{m-b}
m_{(N)}^{(k)} = \frac{s(N,N-k+1)}{(N-1) \cdots (N-k+1)} \,
m_{(N)}^{(1)}
\end{equation}
for $2 \le k \le N$. Note that \eqref{m-b} (for $2 \le k \le N-1$)
can be regarded as the special case $J = (0)$ of \eqref{m-a}.

Here, $s(n,m)$ are the Stirling numbers of the first kind. These are
defined by the generating functions
\begin{equation}\label{S-GF}
\sum_{k=0}^n s(n,k) x^k = x(x\!-\!1)\cdots(x\!-\!n\!+\!1) = b_n(x).
\end{equation}
In particular, we have
\begin{equation}\label{s-ex}
s(n,1) = (-1)^{n-1} (n\!-\!1)!, \quad s(n,n\!-\!1) = - \binom{n}{2}
\quad \mbox{and} \quad s(n,n) = 1.
\end{equation}

Note that the definitions show that
\begin{equation}\label{m-ex}
m_{(a,J)}^{(2)} = \frac{s(N,N) |J| + s(N,N-1)}{N\!-\!1} \,
m_{(a,J)}^{(1)} = \left( \frac{N\!-\!a}{N\!-\!1} -
\frac{N}{2}\right) m_{(a,J)}^{(1)}
\end{equation}
if $a+|J|=N$, and
$$
m_{(N)}^{(2)} = \frac{s(N,N\!-\!1)}{N\!-\!1} = - \frac{N}{2}
m_{(N)}^{(1)}.
$$

Finally, we use the operators
\begin{equation}\label{C1}
\CC_{2N}^{(k)} = \sum_{|I|=N} m_I^{(k)} P_{2I} \quad \mbox{for $1
\le k \le N-1$}
\end{equation}
and
\begin{equation}\label{C2}
\CC_{2N}^{(N)} = (-1)^{N-1} P_{2N}
\end{equation}
to define the operator-valued polynomials
\begin{equation}\label{pol}
\pi_{2N}(\lambda) = \sum_{k=1}^N \CC_{2N}^{(k)} \frac{1}{(N\!-\!k)!}
\left(\lambda\!+\!\f\!-\!N\right)^{N-k}, \; N \ge 1.
\end{equation}

We display explicit formulas for these polynomials for $N \le 3$.

\begin{exs}\label{P-ex} We have, $\pi_2(\lambda) = P_2$,
$$
\pi_4(\lambda) = (P_4-P_2^2) \left(\lambda\!+\!\f\!-\!2\right)- P_4
$$
and
\begin{multline*}
\pi_6(\lambda) = (P_6 - 2P_4P_2 - 2 P_2 P_4 + 3 P_2^3) \frac{1}{2!}
\left(\lambda\!+\!\f\!-\!3\right)^2 \\
+ \left(-\frac{3}{2} P_6 + 2P_4P_2 + P_2 P_4 - \frac{3}{2} P_2^3
\right) \left(\lambda\!+\!\f\!-\!3\right) + P_6.
\end{multline*}
\end{exs}

Note that $m_I = m_{I^{-1}}$, where $I^{-1}$ is the inverse
composition of $I$. Since the GJMS-operators are formally
self-adjoint (see \cite{GZ}), this fact implies that
$\CC_{2N}^{(1)}$ is formally self-adjoint, too.

The main result of the present section consists in the following
characterization of the polynomials $\pi_{2N}(\lambda)$.

\begin{thm}\label{factor} For any $N \ge 1$, the polynomial
$\pi_{2N}(\lambda)$ satisfies the $N$ identities
\begin{equation}\label{fact-a}
\pi_{2N}\left(-\f\!+\!2N\!-\!j\right) = (-1)^j P_{2j} \,
\pi_{2N-2j}\left(-\f\!+\!2N\!-\!j\right), \; j=1, \dots,N-1
\end{equation}
and
\begin{equation}\label{fact-b}
\pi_{2N}\left(-\f\!+\!N\right) = (-1)^{N-1} P_{2N}.
\end{equation}
\end{thm}

Since $\pi_{2N}(\lambda)$ has degree $N-1$, the factorizations
\eqref{fact-a} and \eqref{fact-b} uniquely determine this polynomial
in terms of the lower-order relatives $\pi_2, \dots, \pi_{2N-2}$ and
the GJMS-operator $P_{2N}$.

As a preparation of the proof of Theorem \ref{factor} we need the
following result.

\begin{lemm}\label{L} For all $N \ge 2$,
\begin{equation}\label{L1}
\sum_{2 \le a+b \le N} s(N,a\!+\!b) x^a y^b = \frac{yx(x\!-\!1)
\cdots (x\!-\!N\!+\!1) - xy(y\!-\!1) \cdots (y\!-\!N\!+\!1)}{x-y}
\end{equation}
for $x \ne y$. Moreover,
\begin{equation}\label{L2}
\sum_{2 \le a+b \le N} s(N,a\!+\!b) M^{a+b-1} = (-1)^{N-M-1} M!
(N\!-\!M\!-\!1)!
\end{equation}
for $M=0,1,\dots,N-1$. Here the sums run over all natural numbers
$a, b \ge 1$ subject to the condition $2 \le a+b \le N$.
\end{lemm}

\begin{proof} The sum in \eqref{L1} can be written in the form
\begin{equation}\label{pl1}
s(N,N) \sum_{a=1}^{N-1} x^a y^{N-a} + s(N,N-1) \sum_{a=1}^{N-2} x^a
y^{N-a-1} + \cdots + s(N,2) xy.
\end{equation}
Now \eqref{pl1} equals
\begin{multline*}
\sum_{k=2}^N s(N,k) \left( y \frac{x^k-y^k}{x-y} - y^k \right) \\
= \frac{y}{x-y} \left( \sum_{k=2}^N s(N,k) x^k - \sum_{k=2}^N s(N,k) y^k \right)
- \sum_{k=2}^N s(N,k) y^k.
\end{multline*}
In view of \eqref{S-GF}, the latter sum simplifies to
\begin{multline*}
\frac{y}{x-y} \left( x(x\!-\!1) \cdots (x\!-\!N\!+\!1)
- y (y\!-\!1)\cdots (y\!-\!N\!+\!1) - s(N,1) x + s(N,1)y \right) \\
- (y(y\!-\!1) \cdots (y\!-\!N\!+\!1) - s(N,1) y).
\end{multline*}
Now \eqref{L1} follows from here by a further simplification.
Finally, \eqref{L2} follows from \eqref{L1} by taking the limit $x
\to y$ for $y=0,1,\dots,N-1$.
\end{proof}

We continue with the

\begin{proof}[Proof of Theorem \ref{factor}] \eqref{fact-b} is
obvious by \eqref{C2}. The left-hand side of \eqref{fact-a} equals
\begin{equation}\label{pt-1}
\sum_{k=1}^{N-1} \CC_{2N}^{(k)} \frac{1}{(N\!-\!k)!} (N\!-\!j)^{N-k}
+ (-1)^{N-1} P_{2N}.
\end{equation}
We prove that all non-trivial contributions to this sum are
multiples of the operators
$$
P_{2j} P_{2J}, \; j+|J|=N.
$$
Moreover, we determine the corresponding weights. Let $J$ be
non-trivial. Eq.~\eqref{m-a} shows that the coefficient of $P_{2j}
P_{2J}$ in
\begin{equation}\label{term}
\sum_{k=1}^{N-1} \CC_{2N}^{(k)} \frac{1}{(N\!-\!k)!} (N\!-\!j)^{N-k}
\end{equation}
is given by
\begin{align*}
& \sum_{k=1}^{N-1} \left(\frac{\sum_{i=0}^{k-1} s(N,N-i)
(N\!-\!j)^{k-1-i}} {(N\!-\!1) \cdots (N\!-\!k\!+\!1)} \right)
\frac{(N\!-\!j)^{N-k}}{(N\!-\!k)!} \, m_{(j,J)}^{(1)} \\
& = \frac{1}{(N\!-\!1)!} \sum_{k=1}^{N-1}
\sum_{i=0}^{k-1} s(N,N\!-\!i) (N\!-\!j)^{N-1-i} \, m_{(j,J)}^{(1)} \\
& = \frac{1}{(N\!-\!1)!} \sum_{2 \le i+k \le N} s(N,i+k)
(N\!-\!j)^{i+k-1} \, m_{(j,J)}^{(1)}.
\end{align*}
Eq.~\eqref{L2} implies that the latter sum equals
\begin{equation}\label{tp-3}
(-1)^{j-1} \frac{(N-j)!(j-1)!}{(N\!-\!1)!} m_{(j,J)}^{(1)} =
(-1)^{j-1} \binom{N\!-\!1}{j\!-\!1}^{-1} m_{(j,J)}^{(1)}
\end{equation}
for $1 \le j \le N-1$.

Next, \eqref{S-GF} shows that $P_{2N}$ contributes to \eqref{pl1}
with the coefficient
\begin{multline*}
(-1)^{N-1} + \sum_{k=1}^{N-1} m_{(N)}^{(k)} \frac{1}{(N\!-\!k)!} (N\!-\!j)^{N-k} \\
= (-1)^{N-1} + \frac{1}{(N\!-\!1)!}  \sum_{k=2}^N s(N,k) (N\!-\!j)^{k-1}
=  (-1)^{N-1} - \frac{s(N,1)}{(N\!-\!1)!} =  0.
\end{multline*}
In the last step we have used \eqref{s-ex}.

Now let $l \ne j$. The coefficient of
$$
P_{2l}P_{2J}, \; l+|J|=N, \, 1 \le l \le N-1
$$
in \eqref{term} is given by
\begin{multline*}
\frac{1}{(N-1)!} \sum_{k=1}^{N-1} \left( \sum_{i=0}^{k-1}
s(N,N-i) (N-l)^{k-1-i} \right) (N-j)^{N-k} m_{(l,J)}^{(1)} \\
= \frac{1}{(N\!-\!1)!} \left( \sum_{2 \le i+k \le N} s(N,i+k)
(N-l)^{i-1} (N-j)^k \right) m_{(l,J)}^{(1)}.
\end{multline*}
Lemma \ref{L} implies that this sum vanishes.

Finally, we prove that the weight of $P_{2j}P_{2J}$ on the left-hand
side of \eqref{fact-a} coincides with its weight on the right-hand
side. For this we write $J = (r,K)$ (with a possibly trivial $K$).
For non-trivial $K$, we have $1 \le r < N-j$. In this case,
\eqref{tp-3} shows that the assertion is equivalent to
\begin{multline}\label{pt-3}
\sum_{k=1}^{N-j-1} \left( \frac{ \sum_{i=0}^{k-1}
s(N-j,N\!-\!j\!-\!i) (N\!-\!j\!-\!r)^{k-1-i}}{(N\!-\!j\!-\!1)
\cdots (N\!-\!j\!-\!k\!+\!1)}\right) \frac{N^{N-j-k}}{(N\!-\!j\!-\!k)!} \\
= - m_{(j,r,K)}^{(1)}/m_{(r,K)}^{(1)} \binom{N\!-\!1}{j\!-\!1}^{-1}.
\end{multline}
Using the abbreviation $M=N-j$, the left-hand side of \eqref{pt-3} equals
\begin{multline*}
\frac{1}{(M\!-\!1)!} \sum_{k=1}^{M-1}
\sum_{i=0}^{k-1} s(M,M-i) (M-r)^{k-1-i} N^{M-k} \\[-3mm]
= \frac{1}{(M\!-\!1)!} \sum_{a+b \le M} s(M,a+b) (M-r)^{a-1} N^b.
\end{multline*}
We apply Lemma \ref{L} to simplify this sum. We find
\begin{multline*}
-\frac{1}{(N\!-\!j\!-\!1)!} \left( \frac{N (M\!-\!r) \cdots (-r\!+\!1)
- (M\!-\!r) N (N\!-\!1) \cdots (N\!-\!M\!+\!1)}{(N\!-\!j\!-\!r)(j\!+\!r)}\right) \\
= \frac{1}{j+r} \frac{N!}{j!(N\!-\!j\!-\!1)!}.
\end{multline*}
The assertion \eqref{pt-3} follows by combining this result with
$$
m^{(1)}_{(j,r,K)}/m^{(1)}_{(r,K)} = - \frac{1}{j+r} \binom{N}{j}^2
\frac{j(N\!-\!j)}{N}.
$$
For trivial $K$, i.e., $J=(r)$ and $r=N-j$, \eqref{tp-3} shows that
the assertion is equivalent to
\begin{equation}\label{last}
\frac{1}{(r\!-\!1)!} \sum_{k=1}^{r-1} s(r,r\!-\!k\!+\!1) N^{r-k+1} +
(-1)^{r-1} = - m_{(j,r)}^{(1)}/m_{(r)}^{(1)}
\binom{N\!-\!1}{j\!-\!1}^{-1}.
\end{equation}
Here, the term $(-1)^{r-1}$ on the left-hand side comes from the
contribution of $\CC_{2r}^{(r)}$. By \eqref{S-GF} and \eqref{s-ex},
the left-hand side equals
$$
\frac{1}{(r\!-\!1)!N} \left[N (N\!-\!1) \cdots (N\!-\!r\!+\!1) +
(-1)^r (r\!-\!1)! N \right] + (-1)^{r-1} = \frac{(N\!-\!1)!}{j!
(N\!-\!j\!-\!1)!}.
$$
On the other hand,
$$
m^{(1)}_{(j,r)}/m^{(1)}_{(r)} = - \frac{1}{N} \binom{N}{j}^2
\frac{j(N\!-\!j)}{N}.
$$
This yields \eqref{last}. The proof is complete.
\end{proof}

\begin{rem}\label{CF} Similar arguments can be used to prove the closed
formula
\begin{equation}\label{closed}
\pi_{2N}(\lambda) = \frac{1}{(N\!-\!1)!} \sum_{|I|=N}
\frac{b_N(\lambda\!+\!\f\!-\!N)}{\lambda\!+\!\f\!-\!2N\!+\!I_l} m_I
P_{2I}
\end{equation}
(see \eqref{S-GF}). Here $I_l$ denotes the most left entry of the
composition $I$. Note that the coefficients in \eqref{closed} are
polynomials of degree $N\!-\!1$ since
$$
\lambda\!+\!\f\!-\!2N\!+\!I_l=\left(\lambda\!+\!\f\!-\!N\right)
\!-\!(N\!-\!I_l)
$$
and the integers $N\!-\!I_l$ are zeros of $b_N(x)$. We omit the
details.
\end{rem}

\section{The recursive structure of $Q$-curvature}\label{curv}

In the present section we prove Theorem \ref{Q-GF}.

The proof utilizes properties of $Q$-curvature polynomials. The
notion of $Q$-curvature polynomials was introduced in
\cite{juhl-book} (see also \cite{BJ}). We briefly recall this
concept. Assume that $n$ is even. Associated to any Riemannian
manifold $(M,g)$ of dimension $n$, there is a finite sequence
$Q_2^{res}(g;\lambda), Q_4^{res}(g;\lambda), \dots,
Q_n^{res}(g;\lambda)$ of polynomials of respective degrees
$1,2,\dots,\f$. These polynomials are defined by the constant terms
\begin{equation}\label{Q-pol}
Q_{2N}^{res}(g;\lambda) = -(-1)^N D_{2N}^{res}(g;\lambda)(1)
\end{equation}
of the so-called residue families
$$
D_{2N}^{res}(g;\lambda): C^\infty([0,\varepsilon) \times M) \to
C^\infty(M). \footnote{The name comes from their relation to a
certain residue construction (see \cite{juhl-book}).}
$$
These are families of local operators which are defined in terms of
the holographic coefficients $v_{2j}$ and the coefficients
$\T_{2j}(\lambda)(f)$ in the asymptotic expansion
$$
u \sim \sum_{j \ge 0} r^{\lambda+2j} \T_{2j}(\lambda)(f), \;
\T_0(\lambda)(f) = f, \; r \to 0
$$
of eigenfunctions
$$
-\Delta_{g_+} u = \lambda(n-\lambda) u.
$$
of the Laplace-Beltrami operator for the Poincar\'e-Einstein metric
$g_+$ corresponding to $g$. The coefficients $\T_{2j}(g;\lambda)$
are meromorphic families (in $\lambda$) of differential operators on
$M$. The residue families are conformally covariant generalizations
of the GJMS-operators in the following sense. For any GJMS-operators
$P_{2N}$, the family $D_{2N}^{res}(\lambda)$ contains $P_{2N}$ in
the sense that
\begin{equation}\label{special}
D_{2N}^{res}\left(g;-\f\!+\!N\right) = P_{2N}(g) i^*,
\end{equation}
where $i: M \hookrightarrow [0,\varepsilon) \times M$ denotes the
embedding $m \mapsto (0,m)$. Moreover, $D_{2N}^{res}(g;\lambda)$ is
conformally covariant in the sense that it satisfies the
transformation law
$$
e^{-(\lambda-2N) \varphi} D_{2N}^{res}(e^{2\varphi}g;\lambda) =
D_{2N}^{res}(g;\lambda) \circ \kappa_* \circ
\left(\frac{\kappa^*(r)}{r}\right)^\lambda
$$
for all $\varphi \in C^\infty(M)$. Here, $\kappa$ denotes the
diffeomorphism which relates the Poincar\'e-Einstein metrics of $g$
and $\hat{g} = e^{2\varphi}g$, i.e.,
\begin{equation*}
\kappa^* \left(r^{-2}(dr^2 + g_r)\right) = r^{-2}(dr^2 + \hat{g}_r)
\end{equation*}
and $\kappa$ restricts to the identity on $r=0$.

In terms of the families $\T_{2N}(\lambda)$ and the holographic
coefficients, the $Q$-curvature polynomials are defined by
\begin{multline*}\label{Q-pol-gen}
Q_{2N}^{res}(h;\lambda) = -2^{2N} N!
\left(\left(\lambda\!+\!\f\!-\!2N\!+\!1\right) \cdots
\left(\lambda\!+\!\f\!-\!N\right)\right) \\ \times
\left[\T_{2N}^*(h;\lambda\!+\!n\!-\!2N)(v_0) + \dots +
\T_0^*(h;\lambda\!+\!n\!-\!2N)(v_{2N})\right].
\end{multline*}
Here, the overall polynomial factor has the effect to remove poles.

An additional important feature of residue families is that they
satisfy a system of factorization identities which generalize
\eqref{special}. In fact, we have
\begin{equation}\label{special-all}
D_{2N}^{res}\left(g;-\f\!+\!2N\!-\!j\right) = P_{2j}(g)
D_{2N-2j}^{res}\left(g;-\f\!+\!2N\!-\!j\right), \; j=1,\dots,N.
\end{equation}
Here, \eqref{special} is contained as the special case $j=N$; note
that $D_0^{res}(g;\lambda) = i^*$. Now \eqref{special-all} implies
that
\begin{equation}\label{factorization}
Q_{2N}^{res}\left(-\f\!+\!2N\!-\!j\right) = (-1)^j P_{2j}
Q_{2N-2j}^{res}\left(-\f\!+\!2N\!-\!j\right), \; j=1, \dots,N.
\end{equation}
Note that $Q_0^{res}(\lambda)=-1$.

For full details on residue families and $Q$-curvature polynomials
see \cite{juhl-book} and \cite{BJ}.

We continue with the

\begin{proof}[Proof of Theorem \ref{Q-GF}] The
assertion is equivalent to
\begin{equation}\label{reduced}
\sum_{a+|J|=N} m_{(J,a)} (-1)^a P_{2J}(Q_{2a}) = 2^{2N} N!(N\!-\!1)!
w_{2N}, \; N \ge 1.
\end{equation}
We prove \eqref{reduced} by comparing two different evaluations of
the leading coefficient of the $Q$-curvature polynomial
$Q_{2N}^{res}(\lambda)$. First, assume that $n$ is odd. On the one
hand, Proposition 4.2 in \cite{Q8} shows that the coefficient of
$\lambda^N$ is
\begin{equation}\label{eval-1}
-2^{2N} N! w_{2N}.
\end{equation}
On the other hand, the degree $N$ polynomial $Q_{2N}^{res}(\lambda)$
satisfies the identities
\begin{equation}\label{fact-c}
Q_{2N}^{res}\left(-\f\!+\!2N\!-\!j\right) = (-1)^j P_{2j}
Q_{2N-2j}^{res}\left(-\f\!+\!2N\!-\!j\right), \; j=1, \dots,N\!-\!1
\end{equation}
and
\begin{equation}\label{fact-d}
Q_{2N}^{res}\left(-\f\!+\!N\right) = - \left(\f\!-\!N\right) Q_{2N}
\end{equation}
(see \eqref{factorization}). Moreover, an analog of Theorem 1.6.6 in
\cite{BJ} for odd $n$ states the vanishing result\footnote{For odd
$n$, the proof even simplifies since the families $\T_{2N}(\lambda)$
are regular at $n\!-\!2N$.}
$$
Q_{2N}^{res}(0)=0.
$$
These results show that \eqref{fact-c} and \eqref{fact-d} are
equivalent to the identities
\begin{equation}\label{fact-c2}
\Q_{2N}^{res}\left(-\f\!+\!2N\!-\!j\right) = (-1)^j P_{2j}
\Q_{2N-2j}^{res}\left(-\f\!+\!2N\!-\!j\right), \; j=1, \dots,N-1
\end{equation}
and
\begin{equation}\label{fact-d2}
\Q_{2N}^{res}\left(-\f\!+\!N\right) = Q_{2N}
\end{equation}
for the polynomials
\begin{equation}\label{Q-reduced}
\Q_{2N}^{res}(\lambda) = \lambda^{-1} Q_{2N}^{res}(\lambda).
\end{equation}
By comparing the relations \eqref{fact-c2} and \eqref{fact-d2} with
\eqref{fact-a} and \eqref{fact-b}, Theorem \ref{factor} implies that
the leading coefficient of $Q_{2N}^{res}(\lambda)$ equals
\begin{equation}\label{eval-2}
- \frac{1}{(N\!-\!1)!} \sum_{|J|+a=N} m_{(J,a)} (-1)^a
P_{2J}(Q_{2a}).
\end{equation}
Now the equality of \eqref{eval-1} and \eqref{eval-2} is equivalent
to the asserted identity \eqref{reduced}. Next, assume that $n$ is
even. Then, under the additional assumption $n \ge 4N$, i.e.,
$-\f+2N \le 0$, the sets
$$
\left\{ -\f\!+\!2N\!-\!j \,|\, j=1,\dots,N \right\} \quad \mbox{and}
\quad \{0\}
$$
are disjoint, and the assertion follows by the same arguments as
above. Thus, for fixed $N$, we have proved \eqref{reduced} in all
dimensions $n \ge 4N$. Now we recall that all quantities in
\eqref{reduced} are given by universal expressions in terms of the
metric, its inverse, the curvature and covariant derivatives thereof
with coefficients that are rational functions in $n$ which are
regular for $n \ge 2N$. As a consequence, the relation
\eqref{reduced} holds true also in the remaining cases $4N > n \ge
2N$ (for even $n$).
\end{proof}

Theorem \ref{Q-GF} is equivalent to
\begin{equation}\label{Q-G2}
\G^2\left(\frac{r^2}{4}\right) = v(r).
\end{equation}
This formulation naturally expresses the contributions of
lower-order holographic coefficients $v_{2j}$ ($2j < 2N$) on the
right-hand side of \eqref{Q-G} in terms of lower-order
GJMS-operators acting on lower-order $Q$-curvatures. In fact,
comparing coefficients in \eqref{Q-G2} yields the relations
$$
2 \Lambda_{2N} + \sum_{j=1}^{N-1} \frac{j(N\!-\!j)}{N}
\binom{N}{j}^2 \Lambda_{2j} \Lambda_{2N-2j} = 2^{2N} N! (N\!-\!1)!
v_{2N},
$$
where
$$
\Lambda_{2M} \st \sum_{|a+|J|=M} m_{(J,a)} (-1)^a P_{2J}(Q_{2a}).
$$
In particular, we find
$$
(Q_4 + P_2(Q_2)) + Q_2^2 = 2! 2^3 v_4
$$
and
$$
(Q_6 + 2 P_2(Q_4) - 2 P_4(Q_2) + 3 P_2^2(Q_2)) + 6(Q_4 + P_2(Q_2))
Q_2 = 2!3!2^5 v_6.
$$

\section{A summation formula on round spheres}\label{sum}

We recall that
\begin{equation}\label{vary-Q}
(d/dt)|_0(e^{2N t \varphi} Q_{2N}(e^{2t\varphi}g)) = (-1)^N
P_{2N}^0(g)(\varphi),
\end{equation}
where $P_{2N}^0$ denotes the non-constant part of $P_{2N}$, i.e.,
$P_{2N}^0 = P_{2N} - P_{2N}(1)$. For the proof of \eqref{vary-Q}, we
differentiate the identity
$$
e^{(\f+N)t\varphi} P_{2N}(e^{2t\varphi}g) (e^{-(\f-N)t\varphi}) =
P_{2N}(g)(1)
$$
(see \eqref{covar}) at $t=0$. Using the decomposition
$$
P_{2N} = P_{2N}^0 + (-1)^N \left(\f\!-\!N\right) Q_{2N},
$$
we find
$$
-\left(\f\!-\!N\right) P_{2N}^0(g)(\varphi) + (-1)^N \left(\f\!-\!N
\right) (d/dt)|_0 (e^{2N t \varphi} Q_{2N}(e^{2t\varphi}g)) = 0.
$$
If $2N \ne n$, it suffices to divide this equation by $\f-N$. In the
critical case $2N=n$, \eqref{vary-Q} follows from the fundamental
identity
$$
e^{n\varphi} Q_n(e^{2\varphi}g) = Q_n(g) + (-1)^\f P_n(g)(\varphi).
$$

Now combining Theorem \ref{Q-GF} with \eqref{vary-Q} implies a
formula for the non-constant part of the self-adjoint operator
\begin{equation}
\M_{2N} \st \CC_{2N}^{(1)} = \sum_{|I|=N} m_I P_{2I}.
\end{equation}
Conjecture 11.1 in \cite{juhl-power} states that the operator
$\M_{2N}$ actually can be identified with a certain {\em
second-order} operator. In particular, a huge cancellation takes
place. Since the sum defining $\M_{2N}$ contains the term $P_{2N}$
(we recall that $m_{(N)} = 1$), this relation can be seen as a
recursive formula which expresses $P_{2N}$ in terms of lower-order
GJMS-operators (and some additional terms). The following summation
formula on round sphere is an important special case.

\begin{thm}[\cite{juhl-power}]\label{sum-1} On the round sphere $\S^n$,
\begin{equation}\label{summation-1}
\sum_{|I|=N} m_I P_{2I} = N! (N\!-\!1)! P_2, \; N \ge 1.
\end{equation}
\end{thm}

For the proof of an analog of Theorem \ref{sum-1} on the conformally
flat pseudo-spheres we refer to \cite{JK}.\footnote{A full proof of
Conjecture 11.1 will appear in \cite{juhl-explicit}.}

The following result provides an analogous summation formula in
which the numbers $m_I = m_I^{(1)}$ are replaced by $m_I^{(2)}$.

\begin{thm}\label{sum-2} On the round sphere $\S^n$,
\begin{equation}\label{summation-2}
\sum_{|I|=N} m_I^{(2)} P_{2I} = -\frac{N!(N\!-\!1)!}{2!} (P_2^2 + N P_2), \; N \ge 1.
\end{equation}
\end{thm}

In particular, Theorem \ref{sum-2} shows that the operator of order
$2N$ on the left-hand side is an operator of order four. We expect
that for general metrics the operator on the left-hand side of
\eqref{summation-2} is an operator of order four, too.

\begin{proof}[Proof of Theorem \ref{sum-2}] The arguments in the proof
of Theorem \ref{sum-1} in \cite{juhl-power} show that
\begin{multline*}
\sum_{a+|J|=N} m_{(a,J)}^{(1)} a P_{2a} P_{2J} \\[-3mm]
= \sum_{s=0}^{N-1} (-1)^{N+s} P_{2(N-s)} \frac{N!(N\!-\!1)!^2}
{(N\!-\!s)!s!(N\!-\!s\!-\!1)!^2} \sum_{a=1}^{N} (-1)^a a
\binom{N\!-\!s\!-\!1}{a\!-\!1};
\end{multline*}
note that the right-hand side of this formula differs from the last
formula in the proof of Lemma 6.3 in \cite{juhl-power} only by the
additional coefficient $a$ in the sum over $a$. Now the summation
formula
$$
\sum_{a \ge 0} (-1)^a (a+1) \binom{n}{a} = \begin{cases} 0 & n \ge 2
\\ -1 & n = 1 \\
1 & n = 0
\end{cases}
$$
implies that, in the above sum, the sum over $a$ vanishes except for
$s=N-1$ and $s=N-2$. These two contributions yield
$$
N! (N\!-\!1)! P_2 + \frac{N\!-\!1}{2} N! (N\!-\!1)! P_4 = N!
(N\!-\!1)! \left( \frac{N\!-\!1}{2} P_2^2 + N P_2 \right)
$$
by using $P_4 = P_2 (P_2+2)$. Now Theorem \ref{sum-1} and
$$
m_{(a,J)}^{(2)} = \left( \frac{N\!-\!a}{N\!-\!1} -
\frac{N}{2}\right) m_{(a,J)}^{(1)}
$$
(see \eqref{m-ex}) show that
$$
\sum_{a+|J|=N} m_{(a,J)}^{(2)} P_{2a} P_{2J}
$$
equals
$$
N!(N\!-\!1)! \left( \left( \frac{N}{N\!-\!1} - \frac{N}{2}\right)
P_2 - \frac{1}{N\!-\!1} \left(\frac{N\!-\!1}{2} P_2^2 + NP_2 \right)
\right).
$$
Simplification yields the assertion.
\end{proof}

\begin{rem} Similar arguments prove the summation formulas
$$
\sum_{|I|=N} m_I^{(3)} P_{2I} = \frac{N!(N\!-\!1)!}{3!2!} (P_2^3 +
(3N\!-\!1)P_2^2 + N(3N\!-\!1)/2 P_2), \; N \ge 1
$$
on $\S^n$.
\end{rem}


\end{document}